\documentclass[12pt, reqno]{amsart}
\usepackage{pifont}
\usepackage{mathrsfs}
\usepackage{geometry,color}
\usepackage{titletoc}
\usepackage{stix2}

\usepackage{amsmath}
\usepackage{amssymb} 
\usepackage{enumitem} 
\usepackage{mathtools} 
\usepackage[table]{xcolor} 
\usepackage[all]{xy} 
\usepackage{tikz} 
\usepackage{tikz-cd}
\usepackage{indentfirst} 
\usepackage{babel} 
\usepackage{setspace} 

\usepackage[colorlinks,linkcolor=red,anchorcolor=green,citecolor=blue]{hyperref} 
\hypersetup{linktocpage = true} 

\usepackage{rotating} 

\usepackage{ytableau} 
\usepackage{longtable} 
\newcolumntype{M}[1]{>{\centering\arraybackslash}m{#1}} 

\usepackage{fancyhdr}

\newcommand{\Rmnum}[1]{\expandafter\@slowromancap\romannumeral #1@}

\theoremstyle{plain}
\newtheorem{thm}{Theorem}[section]
\newtheorem{lem}[thm]{Lemma}
\newtheorem{prop}[thm]{Proposition}
\newtheorem{cor}[thm]{Corollary}
\newtheorem{defn}[thm]{Definition}

\newtheorem{conjecture}[thm]{Conjecture}

\theoremstyle{definition}
\newtheorem{example}[thm]{Example}
\newtheorem{rem}[thm]{Remark}

\newcommand{\btheorem}{\begin{thm}}
    \newcommand{\etheorem}{\end{thm}}
\newcommand{\bproposition}{\begin{prop}}
    \newcommand{\eproposition}{\end{prop}}
\newcommand{\bdefinition}{\begin{defn}}
    \newcommand{\edefinition}{\end{defn}}
\newcommand{\bcorollary}{\begin{cor}}
    \newcommand{\ecorollary}{\end{cor}}
\newcommand{\bproof}{\begin{proof}}
    \newcommand{\eproof}{\end{proof}}
\newcommand{\bremark}{\begin{remark}}
    \newcommand{\eremark}{\end{remark}}
\newcommand{\eexample}{\end{example}}
\newcommand{\bexample}{\begin{example}}

\newcommand{\elemma}{\end{lemma}}
\newcommand{\blemma}{\begin{lemma}}

\renewcommand{\bar}{\overline}

\renewcommand{\phi}{\varphi}

\newcommand{\beq}{\begin{equation}}
\newcommand{\eeq}{\end{equation}}
\newcommand{\ee}{\end{eqnarray*}}
\newcommand{\be}{\begin{eqnarray*}}

\newcommand{\bd}{\begin{enumerate}}
    \newcommand{\ed}{\end{enumerate}}

\renewcommand{\hat}{\widehat}
\renewcommand{\tilde}{\widetilde}

\newcommand{\thmref}[1]{Theorem~\ref{#1}}

\usepackage{harmony}

\renewcommand{\#}{\sharp}

\newcommand{\tr}{\mathrm{tr}}
\newcommand{\norm}[1]{\left\lVert #1\right\rVert}

\setlist[itemize]{leftmargin=*}
\setlist[enumerate]{leftmargin=*}

\numberwithin{equation}{section} 

\makeatletter

\usepackage{fancyhdr}
\title{Diffusion and Reaction in the Quasi-Spherical Equation: Mean Curvature Deformations for Fill-Ins}

\author{Wenlong Wang}
\author{Guodong Wei}

\address{Wenlong Wang, School of Mathematical Sciences and LPMC, Nankai University, Tianjin, 300071, P. R. China.}
\email{wangwl@nankai.edu.cn}
\address{Guodong Wei, School of Mathematics (Zhuhai), Sun Yat-sen University, Zhuhai, Guangdong, 519082, P. R. China}
\email{weigd3@mail.sysu.edu.cn}

\begin{document}
\raggedbottom

\begin{abstract}  
We study the diffusion and reaction effects of Bartnik’s quasi-spherical equation to deform boundary mean curvature in fill-in problems with scalar curvature bounded below. The diffusion effect yields an explicit \(L^p\)-to-\(L^\infty\) estimate for \(\partial_tu=u^2\Delta u\), thereby extending the known upper bound for the minimum boundary mean curvature of fill-ins to a quantitative upper bound for its harmonic mean. For spin fill-ins, this bound is explicit and involves only coarse intrinsic boundary data. By introducing an absorbing reaction term, we also construct a deformation that transforms any nonnegative initial mean curvature into a terminal mean curvature with a uniform positive lower bound. For Gromov’s conjecture on total mean curvature, this reduces the \(H\geq 0\) case to Theorem A of Frenck, Hanke, and Hirsch \cite{FHH}, which assumes \(H\geq\kappa>0\). This covers the case of non-spin boundaries, complementing their result for spin boundaries (Theorem B) for this conjecture.  
\end{abstract}

\maketitle

\section{Introduction}
Scalar curvature plays an important role in both differential geometry and general relativity. Geometrically, it is one of the fundamental curvature invariants of a Riemannian manifold and has profound implications for the global geometry and topology of the underlying manifold. In general relativity, for a time-symmetric initial data set, the dominant energy condition reduces precisely to the nonnegativity of the scalar curvature. A cornerstone illustrating this connection is the positive mass theorem (see \cite{BHH, BW, Sch89,SY79a,SY17, Witten}), which asserts that a complete asymptotically flat manifold with nonnegative scalar curvature has nonnegative ADM mass, with equality if and only if the manifold is isometric to Euclidean space.

The study of scalar curvature geometry on manifolds with boundary is closely related to the theory of quasi-local mass (see \cite{Ba1,BY,Ha2, LY, WY}), which seeks to measure the gravitational energy contained within a bounded region in terms of the intrinsic geometry and mean curvature of its boundary. This naturally leads to the notion of Bartnik data, namely, a triple
\(
(\Sigma,\gamma,H)
\), where $(\Sigma,\gamma)$ is a closed oriented Riemannian manifold that is null-cobordant and $H$ is a prescribed function representing the boundary mean curvature.

A fundamental problem is to determine whether prescribed Bartnik data admit a fill-in with a given scalar curvature lower bound. More generally, one may consider deformations of Bartnik data and ask whether two such data sets can be joined by a Riemannian cobordism with a given scalar curvature lower bound. Such constructions provide a flexible approach to studying scalar curvature geometry. For example, such a cobordism played an essential role in the solution of Gromov's positive scalar curvature extension problem by Shi and the present authors \cite{SWW}. Such a cobordism was also used by Weinberger, Xie, and Yu \cite{WXY} to produce counterexamples to Gromov's compactness question for metrics of uniformly positive scalar curvature on noncompact manifolds.

In this paper, we study deformations of Bartnik data together with their applications to fill-in problems. For these applications, we restrict attention to Bartnik data on the same underlying manifold and to cobordisms of the form 
\(\Sigma\times[0,1]\), which we call necks. We therefore adopt the following restricted notion of cobordism. Throughout the paper, unless otherwise stated, boundary mean curvature is computed with respect to the inward unit normal; in particular, the boundary of a Euclidean ball has positive mean curvature.
\begin{defn}
Let \(\Sigma\) be a closed manifold, let \(\gamma_0\) and \(\gamma_1\) be smooth Riemannian metrics on \(\Sigma\), and let \(H_0\), \(H_1\in C^\infty(\Sigma)\). A smooth Riemannian metric \(g\) on \(\Sigma\times[0,1]\) is called a cobordism from \((\gamma_0,H_0)\) to \((\gamma_1,H_1)\) if
\[
g|_{\Sigma\times\{0\}}=\gamma_0,\qquad
g|_{\Sigma\times\{1\}}=\gamma_1,
\]
and
\[
H_g|_{\Sigma\times\{0\}}=H_0,\qquad
H_g|_{\Sigma\times\{1\}}=H_1.
\]
\end{defn}
We first study whether mean curvature that is large in an integral sense can be deformed to become large pointwise, and obtain the following result.

\begin{thm}\label{cobordant1}
Let $(\Sigma^n,\gamma)$ be a closed Riemannian manifold. For any $p>0$ and $\kappa>0$, there exists a constant
$C=C(n,p,\kappa,C_S,V)>0$ such that if a smooth positive function $H$ on $\Sigma$ satisfies
\[
\norm{H}_{L^{-p}}
=
\left(\int_\Sigma H^{-p}\,d\mu_\gamma\right)^{-\frac1p}
\geq C,
\]
then there exist a smooth function $\widetilde H\geq\kappa$ on $\Sigma$
and a cobordism $g$ from $(\gamma,-H)$ to
$\left(2^{4/(n+1)}\gamma,\widetilde H\right)$ such that
\[
R_g\geq
\min\left\{
\min_\Sigma R_\gamma,\,
2^{-\frac{4}{n+1}}\min_\Sigma R_\gamma
\right\}.
\]
Here, $R_g$ denotes the scalar curvature of $g$, and $C_S$ and $V$ denote the Sobolev constant and the volume of $(\Sigma,\gamma)$, respectively. The choice of $C$ is described in the proof.
\end{thm}

If the prescribed scalar curvature is strictly less than that of the slice metric, our construction yields a uniform positive lower bound for the terminal mean curvature, independent of the initial mean curvature.
\begin{thm}\label{cobordant2}
Let $(\Sigma^n,\gamma)$ be a closed Riemannian manifold. For any constant $\delta<\min_{\Sigma} R_{\gamma}$ and any smooth positive function $H$ on $\Sigma$, there exist a smooth function $\widetilde H$ on $\Sigma$ and a cobordism $g$ from $(\gamma,-H)$ to $\left(2^{4/(n+1)}\gamma,\widetilde H\right)$ such that 
\[
\widetilde H\geq \left[
 \frac{n}{4(n-1)}
 \left(2^{\frac{2(n-1)}{n+1}}-1\right)
 \left(\min_{\Sigma} R_{\gamma}-\delta\right)\right]^{\frac12}
,\quad\,  \text{and} \quad\,
R_g\geq\min\left\{\delta,2^{-4/(n+1)}\delta\right\}.
 \] 
 Moreover,
\[
 \int_\Sigma H\,d\mu_\gamma
 \leq 2\int_\Sigma \widetilde H\,d\mu_{\gamma}.
\]
\end{thm}

The cobordisms required in Theorems \ref{cobordant1} and \ref{cobordant2} are constructed using the quasi-spherical metrics introduced by Bartnik \cite{Bartnik1}. In this construction, prescribing the scalar curvature leads to the following equation (see Section \ref{3} for details): 
\begin{equation}\label{qseq}
\bar H_t\frac{\partial u}{\partial t}=u^2\Delta_{\gamma_t}u+\frac{1}{2}\left(\varphi-R_{\gamma_t}\right)u^3+\frac{1}{2}\left(R_{\gamma_t}-R_{\bar g}\right)u.
\end{equation}
Here, \(\varphi\) denotes the prescribed scalar curvature, and \(u\) is the ratio of the mean curvature of a slice with respect to the background metric to that with respect to the constructed metric.

The proof of Theorem \ref{cobordant1} relies on the {\it diffusion effect} of equation \eqref{qseq}, which is captured by the model equation
\begin{equation}\label{maineq}
u_t=u^2\Delta u.
\end{equation}
This falls into a broad class of equations
\begin{equation}\label{fde}
    \partial_t u=u^m\Delta u,
\end{equation}
which arise in a variety of physical models and applications, including groundwater flow through fissured porous media, biological diffusion processes, and heat distribution in solid crystalline molecular hydrogen. We refer the reader to \cite{BBCP, BV, BH, BF, JX, Rosen, U, V} and the references therein for further discussions of these applications.

Equation \eqref{fde} is also closely related to several classical diffusion equations and geometric evolution problems. If $u$ is positive and $m\neq1$, then the transformation
\(v=u^{1-m}\)
yields
\begin{equation}\label{817}
    \partial_t v
    =(1-m)\Delta\big(v^{\frac{1}{1-m}}\big).
\end{equation}
When $0<m<1$, \eqref{817} is the porous-medium equation. When $m>1$, this equation is a generalized filtration equation. When \(m=2\), \eqref{817} corresponds to the \(J\)-flow in complex dimension one, arising from the moment-map framework introduced by Donaldson \cite{D}. When \(n\geq 3\) and \(m=-4/(n-2)\), \eqref{fde} captures the leading-order part of the Yamabe flow equation for the conformal factor $u$. In the formal limit \(m\to-\infty\), \eqref{817} becomes the logarithmic diffusion equation
\(\partial_t v=\Delta\log v,\)
which describes the two-dimensional Ricci flow in local isothermal coordinates.

The key analytic ingredient in the proof of Theorem \ref{cobordant1} is the following $L^p$-to-$L^\infty$ estimate for positive solutions of \eqref{fde}.
\begin{thm}\label{nm}
Let \((\Sigma^n,\gamma)\) be a closed Riemannian manifold with \(n\geq2\), and let \(C_S\) and \(V\) denote its Sobolev constant and volume, respectively. For \(m\geq0\), suppose that \(u\) is a smooth positive solution of
\[
\begin{cases}
\partial_t u=u^m\Delta_\gamma u
& \text{on }\,\Sigma\times[0,T],\\
u(\cdot,0)=u_0>0.
\end{cases}
\]
When $n\geq3$, for every
$p_0>\max\{0,1-m\}$ and $t\in(0,T]$,
\[
\begin{aligned}
\|u(\cdot,t)\|_{L^\infty}
\leq{}&
C(n,m,p_0,C_S)\,
t^{-\frac{n}{2p_0+mn}}
\left(
1+
\left(\frac{2^n}{V}\right)^{\frac{2p_0+mn}{p_0n}}
C_S\|u_0\|_{L^{p_0}}^m t
\right)^{\frac{n+2}{2p_0+mn}}
\|u_0\|_{L^{p_0}}^{\frac{2p_0}{2p_0+mn}},
\end{aligned}
\]
where
\[
C(n,m,p_0,C_S)
=
2^{\frac{(n+2)^2}{2(2p_0+mn)}}
C_S^{-\frac{n}{2p_0+mn}}
\exp\left[
\frac{n+1}{4p_0}
\left(
\frac{m}{p_0}+\frac{1}{p_0+m-1}
\right)
\right].
\]
When $n=2$, the same estimate holds with $n$ replaced by $3$ in both the estimate and the formula for \(C\), and with \(C_S\) taken to be the Sobolev constant for the embedding $W^{1,2}(\Sigma)\hookrightarrow L^6(\Sigma)$.
\end{thm}

The proof of Theorem \ref{nm} relies on a Nash–Moser type iteration scheme. Owing to the structure of \eqref{fde}, however, the iteration differs from the standard scheme for the heat equation. A related result was obtained by Grillo and Muratori in \cite{GM}, where similar $L^{p_0}$-to-$L^\infty$ estimates were established for the porous medium equation $u_t=\Delta(u^m)$ with homogeneous Neumann boundary conditions, under the assumptions $m>1$ and $p_0\geq 1$.

For Theorem \ref{cobordant2}, we derive an estimate by comparison with the corresponding ODE, relying primarily on the {\it reaction effect} of \eqref{qseq} arising from the negative coefficient of \(u^3\).

Theorems \ref{cobordant1} and \ref{cobordant2} are motivated in part by fill-in problems with scalar curvature bounded below, as posed by Gromov \cite{Gro96,Gro18}. Let \((\Sigma,\gamma)\) be a closed orientable Riemannian manifold,  and let $H$ be a smooth function on $\Sigma$. A compact Riemannian manifold \((\Omega,g)\) with boundary is called a {\it fill-in} of \((\Sigma,\gamma,H)\) if there exists an isometry
\[
X:(\Sigma,\gamma)\longrightarrow (\partial\Omega,g|_{\partial\Omega})
\]
such that
\(
H=H_g\circ X
\)
on \(\Sigma\), where \(H_g\) denotes the mean curvature of \(\partial\Omega\) in \((\Omega,g)\). In \cite{Gro18}, Gromov proposed the following conjecture:

\begin{conjecture}[Gromov \cite{Gro18}]\label{conj1}
Let \(\sigma\) be a constant. If \((\Sigma,\gamma,H)\) admits a fill-in with scalar curvature \(R\geq \sigma\), then
\[\min_{\Sigma} H\leq \Lambda,\]
where \(\Lambda\) is a constant depending only on \(\sigma\) and the intrinsic geometry of \((\Sigma,\gamma)\).
\end{conjecture}

Conjecture \ref{conj1} predicts that sufficiently large boundary mean curvature rules out fill-ins with a prescribed lower bound on scalar curvature. Related results had previously been obtained; see, for example, \cite{Miao02,SWW,SWWZ}. It was eventually resolved in full by Miao \cite{Miao21}, who combined the extension result of \cite{SWW} with the theorem of Schoen and Yau \cite{SY792,SY17} asserting that $\mathbf T^n\# M^n$ admits no metric of positive scalar curvature for any closed orientable manifold \(M^n\).

Inspired by these results, we consider the following invariant of \((\Sigma,\gamma)\). For \(\sigma\in\mathbb R\), define 
\[
\eta(\Sigma,\gamma,\sigma)=\sup\left\{\,\min_{\Sigma} H\,\big|\,(\Sigma,\gamma,H)\text{ admits a fill-in with scalar curvature $R\geq\sigma$}\,\right\}.
\]
Miao’s result can then be expressed as $\eta(\Sigma,\gamma,\sigma)<+\infty$. It is therefore natural to seek quantitative upper bounds for this invariant and to investigate rigidity in the equality case. This line of research has been fruitful for spin fill-ins; see, for example, \cite{AL,B0,BTW,CHZ,Gro18,HMR,HMZ}. In particular, Gromov \cite{Gro18} proved that if the Bartnik data $(\Sigma^n,\gamma,H)$ admit a spin fill-in with nonnegative scalar curvature, then
\begin{equation*}
\min_{\Sigma}H
\leq
\frac{n}{\mathrm{Rad}(\Sigma,\gamma)}.
\end{equation*}
Here, $\mathrm{Rad}(\Sigma,\gamma)$ denotes the hyperspherical radius introduced by Gromov and Lawson \cite{GL80}, defined by
\begin{equation*}
\mathrm{Rad}(\Sigma,\gamma)
=
\sup\left\{
\,r>0
\,\middle|\,
\begin{array}{c}
\text{there exists a $1$-Lipschitz map}\\
(\Sigma,\gamma)\longrightarrow \mathbf{S}^n(r)
\text{ of nonzero degree}
\end{array}
\right\}.
\end{equation*}
More recently, combining the Hijazi--Montiel--Rold\'an inequality \cite{HMR} with B\"ar's eigenvalue estimate \cite{B0}, Brendle, Tsiamis, and Wang \cite{BTW} obtained an analogous sharp estimate for spin fill-ins whose scalar curvature is bounded below by $-n(n+1)$. Without the spin assumption, Wang, Wang, and Zhu \cite{WWZ} established Gromov’s estimate and rigidity in the equality case for nonnegative scalar curvature fill-ins of $3$-dimensional manifolds.

As an application of Theorem \ref{cobordant1}, we obtain the following estimate relating uniform positive lower bounds for \(\int_\Sigma H^{-p}\,d\mu_\gamma\) to \(\eta(\Sigma,\gamma,\sigma)\).

\begin{thm}\label{finiteness1} 
Let \((\Sigma^n,\gamma)\) be a closed orientable Riemannian manifold with \(n\geq2\), and let $H$ be a smooth positive function on $\Sigma$. Fix a constant \(\sigma\leq 0\) and assume that \((\Sigma^n,\gamma,H)\) admits a fill-in with scalar curvature \(R\geq\sigma\). If \(n\geq3\), then for every \(p>0\), 
\begin{equation}\label{harmonicmean}
\int_{\Sigma}\frac{1}{H^{p}} \,d\mu_{\gamma}\geq C(n,p)\min\Bigg\{C_S^{\frac{n}{2}}B^{-(p+n)},\left(C_{S}\left(\frac{2^n}{V}\right)^{
\frac{(n+2)(p+n)}{pn}}B^{p+n}\right)^{-\frac{p}{p+n+2}}
\Bigg\}.
\end{equation}
Here $C_S$ and $V$ denote the Sobolev constant and the volume of $(\Sigma,\gamma)$, respectively, and
\[
B=\tau(n)\eta\left(\Sigma,
\gamma,n^{-2}\tau(n)^{-2}\min\big\{\sigma,\min_{\Sigma}R_{\gamma}\big\}\right).
\]
The constants \(C(n,p)\) and \(\tau(n)\) depend only on the indicated parameters, and their explicit expressions are given in \(\eqref{99}\) and \(\eqref{910}\), respectively. When \(n=2\), the same bound holds with \(n\) replaced by \(3\), except that \(B\) is still defined using the actual dimension \(n=2\).
\end{thm}

Recall the classical Heintze--Karcher--Ros inequality \cite{HK,ro}: if $(\Omega^{n+1},g)$ is a fill-in of $(\Sigma,\gamma,H)$ with nonnegative Ricci curvature and \(H>0\), then
\[
\int_{\Sigma}\frac{1}{H}\,d\mu_{\gamma}
\geq
\frac{n+1}{n}\operatorname{Vol}(\Omega).
\]
The \(p=1\) case of Theorem \ref{finiteness1} may therefore be viewed as a scalar curvature analogue of the Heintze--Karcher--Ros inequality.

For spin fill-ins, combining the sharp upper bound for \(\min_\Sigma H\) due to Gromov \cite{Gro18} and to Brendle, Tsiamis, and Wang \cite{BTW} with the argument of Theorem \ref{finiteness1}, we obtain the following lower bound for $\int_{\Sigma} H^{-p}\,d\mu_\gamma$ in terms of explicit boundary geometric data.

\begin{cor}\label{ros}
Let \((\Sigma^n,\gamma)\) be a closed orientable Riemannian manifold with \(n\geq2\), and let $H$ be a smooth positive function on $\Sigma$. Fix a constant \(\sigma\leq0\) and assume that \((\Sigma^n,\gamma,H)\) admits a spin fill-in with scalar curvature \(R\geq\sigma\). If \(n\geq3\), then for every \(p>0\), \eqref{harmonicmean} holds with
\[B=\left(\frac{-\min\left\{\sigma, \min_{\Sigma}\mathrm{R}_{\gamma}\right\}}{n(n+1)}+\mathrm{Rad}(\Sigma,\gamma)^{-2}\right)^{\frac{1}{2}}.\]
When \(n=2\), the same bound holds with \(n\) replaced by \(3\), except that \(B\) is still defined using the actual dimension \(n=2\).
 \end{cor}

Theorem \ref{finiteness1} shows that sufficiently large boundary mean curvature, in the harmonic mean sense, rules out fill-ins with a prescribed lower bound on scalar curvature. In fact, Gromov had already conjectured the stronger statement that sufficiently large total mean curvature does.

\begin{conjecture}[Gromov \cite{Gro23}]\label{conj}
Let \((\Sigma,\gamma)\) be a closed orientable Riemannian manifold, and let $H$ be a smooth positive function on $\Sigma$. Fix \(\sigma\in\mathbb R\) . If \((\Sigma,\gamma,H)\) admits a fill-in with scalar curvature \(R\geq \sigma\), then
\[\int_{\Sigma}H d\mu_{\gamma}\leq\Lambda,\]
where \(\Lambda\) is a constant depending only on \(\sigma\) and the intrinsic geometry of \((\Sigma,\gamma)\).
\end{conjecture}

The celebrated nonnegativity of the Brown--York mass, established by Shi and Tam \cite{ST02}, implies that Conjecture \ref{conj} holds for $(\mathbf S^2,\gamma)$ when $\gamma$ has positive Gaussian curvature, $H>0$, and $\sigma=0$. Subsequently, building on the work of Shi and Tam \cite{ST07}, Mantoulidis and Miao \cite{MM} established the same conclusion for arbitrary smooth metrics on $\mathbf S^2$. Shi and the present authors \cite{SWW} extended this result to arbitrary smooth metrics on \(\mathbf S^n\) for \(n\leq6\) and arbitrary \(\sigma\). Combined with the recent asymptotically hyperbolic positive mass theorem \cite{HKLZ, Tsang}, their argument also applies when \(n\geq7\).

For $\mathbf{T}^2$ equipped with an arbitrary flat metric, Chen, Liu, Shi, and Zhu \cite{CLSZ} proved Conjecture \ref{conj} in the case \(\sigma=0\), assuming \(H>0\). For \(2\leq n\leq6\), Wang \cite{WangY} established the corresponding bound for arbitrary metrics on \(\mathbf T^n\) and general scalar curvature lower bounds, assuming \(H>0\) and additionally requiring the fill-in to be diffeomorphic to \(\mathbf D^2\times\mathbf T^{n-1}\). More recently, B\"ar \cite{B} proved Conjecture \ref{conj} for spin fill-ins. His estimate also allows \(H\) to take negative values, with additional dependence on a lower bound for \(H\). However, his very recent work \cite{B1} shows that this additional dependence is necessary in general for fill-ins of dimension at least three. Independently, Frenck, Hanke, and Hirsch \cite{FHH} made substantial progress by using surgery to reduce the problem to the case of spherical boundaries. More precisely, they proved the following theorem.

\begin{thm}[Frenck--Hanke--Hirsch \cite{FHH}]\label{fhh}
Let $(\Sigma^n,\gamma)$ be a closed Riemannian manifold and let
$\sigma,\kappa\in\mathbb{R}$. Then there exists a constant
$\Lambda=\Lambda(\Sigma,\gamma,\sigma,\kappa)>0$ with the following
property: if $(\Omega,g)$ is a fill-in of $(\Sigma,\gamma,H)$ such that
\[
H\geq\kappa,\qquad R_{g}\geq\sigma,
\]
and one of the following conditions holds:
\begin{enumerate}
\item[(A)] $\kappa>0$ and $n\leq6$;
\item[(B)] $\kappa=0$, $n\leq6$, and $\Sigma$ admits a spin structure;
\item[(C)] $\Omega$ admits a spin structure,
\end{enumerate}
then
\[
\int_\Sigma H\,d\mu_\gamma\leq\Lambda.
\]
\end{thm}

Since part (A) reduces to the spherical case in \cite{SWW}, the recent asymptotically hyperbolic positive mass theorem \cite{HKLZ, Tsang} removes the dimension restriction; see also \cite[Remark 1.2]{FHH}.

Using Theorem \ref{cobordant2}, we reduce the \(H\geq0\) case to part (A) of Theorem \ref{fhh}. This removes the spin assumption on \(\Sigma\) in part (B) and yields the following result.

\begin{cor}\label{finiteness2}
Conjecture \ref{conj} holds for \(H\geq0\), without any spin assumption.
\end{cor} 

Although the total mean curvature bounds of Frenck, Hanke, and Hirsch \cite{FHH} and, for spin fill-ins, of B\"ar \cite{B} already imply, by H\"older’s inequality, a uniform positive lower bound for \(\int_\Sigma H^{-p}\,d\mu_\gamma\), our Theorem \ref{finiteness1} gives a quantitative estimate in terms of the invariant \(\eta\). For spin fill-ins, our Corollary \ref{ros} leads to a fully explicit bound depending only on \(n\), \(p\), coarse intrinsic boundary data, and the prescribed scalar curvature lower bound.

We now outline the main ideas of the proofs of \thmref{cobordant1} and \thmref{cobordant2}. A crucial ingredient in both proofs is the construction of a suitable path of metrics by solving an ODE (see Section \ref{3} for details). This choice of metric path allows the quasi-spherical metric equation \eqref{qseq}, after a suitable reparametrization of time, to be transformed exactly into an equation of the form \eqref{fde}. Once this reduction is achieved, Theorem \ref{nm} allows us to convert largeness of the mean curvature in the harmonic mean sense into pointwise largeness. For Theorem \ref{cobordant2}, we use the same metric path and introduce an absorbing reaction term into the quasi-spherical equation. This yields a neck whose terminal mean curvature has a positive lower bound independent of the initial mean curvature.

\medskip

{\it Organization.}
The remainder of this paper is organized as follows. In Section \ref{2}, we prove Theorem \ref{nm}. Section \ref{3} is devoted to the proofs of Theorem \ref{cobordant1}, Theorem \ref{finiteness1} and Corollary \ref{ros}. Finally, we establish Theorem \ref{cobordant2} and Corollary \ref{finiteness2} in Section \ref{4}.

\section{Proof of Theorem \ref{nm}}\label{2}
In this section, we use Nash--Moser iteration to prove \thmref{nm}. We first specify the Sobolev constant used below. Let \((\Sigma^n,\gamma)\) be a closed Riemannian manifold. For \(n\geq3\), we define its Sobolev constant by
\[C_S=\inf\left\{\frac{\int_\Sigma|\nabla f|^2\,d\mu_{\gamma}}{\left(\int_\Sigma |f|^{2\chi}\,d\mu_{\gamma}\right)^{\frac{1}{\chi}}}\ \Bigg|\ f\in W^{1,2}(\Sigma,\gamma)\setminus\{0\}\quad\text{with}\ \, \int_\Sigma f\,d\mu_{\gamma}=0\right\},\]
where $\chi=n/(n-2)$. Let
$$\bar{f}=\frac{1}{V}\int_{\Sigma} f\,d\mu_{\gamma},$$
where $V$ is the volume of $(\Sigma,\gamma)$. Then
\begin{equation*}
\begin{split}
\|f\|_{L^{2\chi}}&\leq\|f-\bar f\|_{L^{2\chi}}+\|\bar f\|_{L^{2\chi}}\\
&=\|f-\bar f\|_{L^{2\chi}}+V^{-\frac{n+2}{2n}}\left|\int_\Sigma f\,d\mu_{\gamma}\right|\\
&\leq \frac{1}{\sqrt{C_S}}\|\nabla f\|_{L^2}+V^{-\frac{1}{n}}\|f\|_{L^2}.
\end{split}
\end{equation*}
Squaring this estimate and using an elementary inequality, we obtain the following Sobolev inequality for every \(f\in W^{1,2}(\Sigma,\gamma)\):
 \begin{equation}\label{sob}
 \frac{C_S}{2}\left(\int_\Sigma |f|^{2\chi}\,d\mu_{\gamma}\right)^{\frac{1}{\chi}}\leq \int_\Sigma |\nabla f|^2\,d\mu_{\gamma}+C_SV^{-\frac{2}{n}}\int_\Sigma f^2\, d\mu_{\gamma}.\end{equation}

When \(n=2\), we introduce \(N=3\) as an “effective dimension” and set \(\chi=N/(N-2)=3\). We then define
\[
C_S
=
\inf_{\substack{f\in W^{1,2}(\Sigma,\gamma)\setminus\{0\}\\
\int_\Sigma f\,d\mu_\gamma=0}}
\frac{\displaystyle\int_\Sigma|\nabla f|^2\,d\mu_\gamma}
{\displaystyle\left(\int_\Sigma|f|^{2\chi}\,d\mu_\gamma\right)^{\frac1{\chi}}}.
\]
The same argument as above yields
\begin{equation}\label{sb2}
\frac{C_S}{2}
\left(\int_\Sigma|f|^{2\chi}\,d\mu_\gamma\right)^{\frac1{\chi}}
\leq
\int_\Sigma|\nabla f|^2\,d\mu_\gamma
+C_S V^{-\frac{2}{N}}\int_\Sigma f^2\,d\mu_\gamma
\end{equation}
for every $f\in W^{1,2}(\Sigma,\gamma)$.

For convenience, we occasionally omit the volume element in integrals. Unless otherwise stated, all integrals are taken with respect to $d\mu_\gamma$. 

\begin{proof}[\bf Proof of Theorem \ref{nm}:] We first consider the case \(n\geq3\). For any $p>\max\{0,1-m\}$, we have
\begin{equation}\label{monotonicity of Lp-m}
\begin{split}
\frac{d}{dt}\int_{\Sigma} u^{p}&=p\int_{\Sigma} u^{p+m-1}\Delta u\\
&=-A(p,m)\int_{\Sigma}\big|\nabla u^{\frac{p+m}{2}}\big|^2\\
&\leq 0 ,
\end{split}
\end{equation}
where 
$$A(p,m)=\frac{4p(p+m-1)}{(p+m)^2}.$$
Hence, setting \(L=\|u_0\|_{L^{p_0}}\), we obtain
\begin{equation}\label{mp}
\norm{u(\cdot,t)}_{L^{p_0}}\leq L\qquad\text{for all}\ \, t\in [0,T].
\end{equation}
Applying the Sobolev inequality \eqref{sob} to \(u^{(p+m)/2}\), we obtain
\begin{equation}\label{fdifi-m}
\begin{split}
\frac{d}{dt}\int_{\Sigma} u^{p}\leq-\frac{1}{2}A(p,m)C_S\left(\,\int_{\Sigma} u^{\chi(p+m)}\right)^{\frac{1}{\chi}}+A(p,m)C_SV^{-\frac{2}{n}}\int_{\Sigma} u^{p+m}.
\end{split}
\end{equation}
For convenience, set
$$C_1=\frac{1}{2}C_S\qquad \text{and} \qquad C_2=C_SV^{-\frac{2}{n}}.$$
Then, \eqref{fdifi-m} can be rewritten as
\begin{equation}\label{1016}
\frac{d}{dt}\int_{\Sigma} u^{p}\leq -A(p,m)C_1\left(\,\int_{\Sigma} u^{\chi(p+m)}\right)^{\frac{1}{\chi}}+A(p,m)C_2\int_{\Sigma} u^{p+m}.
\end{equation}
 By H\"older's inequality and Young's inequality, we have
\begin{equation*}
\begin{split}
\int_{\Sigma} u^{p+m}&\leq\left(\,\int_{\Sigma} u^{p_0}\right)^\frac{2m}{mn+2p_0}\left(\,\int_{\Sigma} u^{p}\right)^\frac{2p_0}{mn+2p_0}\left(\,\int_{\Sigma} u^{\chi(p+m)}\right)^{\frac{m(n-2)}{mn+2p_0}}\\
&\leq\left(L^m\int_{\Sigma} u^{p}\right)^\frac{2p_0}{mn+2p_0}\left(\,\int_{\Sigma} u^{\chi(p+m)}\right)^{\frac{1}{\chi}\cdot\frac{mn}{mn+2p_0}}\\
&\leq\frac{C_1}{2C_2}\left(\,\int_{\Sigma}  u^{\chi(p+m)}\right)^{\frac{1}{\chi}}+\frac{2p_0}{mn+2p_0}\left(\frac{2mn}{mn+2p_0}\right)^\frac{mn}{2p_0}\left(\frac{C_2}{C_1}\right)^\frac{mn}{2p_0}L^m\int_{\Sigma} u^p.
\end{split}
\end{equation*}
Note that $A(p,m)<4$. It follows that 
\begin{equation}\label{fdifi}
\begin{split}
\frac{d}{dt}\int_{\Sigma} u^{p}\leq-\frac{1}{2}A(p,m)C_1\left(\,\int_{\Sigma}u^{\chi(p+m)}\right)^{\frac{1}{\chi}}+C_3L^m\int_{\Sigma} u^{p},
\end{split}
\end{equation}
where
\begin{equation*}
C_3=2^{\frac{2p_0+nm}{p_0}}C_SV^{-\frac{m}{p_0}-\frac{2}{n}}
=C_S\left(\frac{2^n}{V}\right)^{\frac{2p_0+mn}{p_0n}}.
\end{equation*}

Now, for $0\leq a<b<T$, define $\phi$ by
\begin{equation*}
\phi(s)=\left\{
\begin{aligned}
&  \ 0, \quad\quad\quad\ 0\leq s\leq a,\\
& \ \frac{s-a}{b-a},\quad a\leq s \leq b,\\
&\ 1,\quad\quad\quad\ s\geq b.
\end{aligned}
\right.
\end{equation*}
Multiplying \eqref{fdifi} by $\phi$ and integrating by parts in time, we obtain, for $t\in[a,T]$,
\begin{align*}
 &\phi(t)\int_{\Sigma} u^{p}(\cdot,t)+\frac{1}{2}A(p,m)C_1\int^t_a\phi(s)\left(\,\int_\Sigma u^{\chi(p+m)}(\cdot,s)\right)^{\frac{1}{\chi}}\,ds\\ 
 \leq& \int^{t}_{a}\phi'(s)\int_{\Sigma} u^{p}(\cdot,s)\,ds+C_3L^m\int^t_a \phi(s)\int_\Sigma u^{p}(\cdot,s)\,ds.
\end{align*}
Consequently, for $t\in(b,T]$,
\begin{equation}
\sup_{s\in [b,t]}\int_{\Sigma} u^{p}(\cdot,s)\leq \left(\frac{1}{b-a}+C_3L^m\right)\int^{t}_{a}\int_\Sigma u^{p}(\cdot,s)\,ds,
\end{equation}
and
\begin{equation}
\int^{t}_{b}\left(\,\int_\Sigma u^{\chi(p+m)}(\cdot,s)\right)^{\frac{1}{\chi}}\,ds\leq 2C^{-1}_1A(p,m)^{-1}\left(\frac{1}{b-a}+C_3L^m\right)\int^{t}_a\int_\Sigma u^{p}(\cdot,s)\,ds.
\end{equation}
Let 
$$\lambda=2-\chi^{-1}=\frac{n+2}{n}.$$ 
By H\"older's inequality, we obtain
\begin{align*}
\int_{\Sigma} u^{\lambda p+m}(\cdot,s)\leq\left(\,\int_{\Sigma} u^{p}(\cdot,s)\right)^{1-\frac{1}{\chi}}\left(\,\int_\Sigma u^{\chi(p+m)}(\cdot,s)\right)^{\frac{1}{\chi}}.
\end{align*}
Combining this with the preceding estimates gives
\begin{equation*}
\begin{split}
\int^{t}_{b}\int_\Sigma u^{\lambda p+m}(\cdot,s)\,ds&\leq \int^{t}_{b}\left(\,\int_\Sigma u^{p}(\cdot,s)\right)^{1-\frac{1}{\chi}}\left(\,\int_\Sigma u^{\chi (p+m)}(\cdot,s)\right)^{\frac{1}{\chi}}\,ds\\
&\leq\left(\sup_{s\in [b,t]}\,\int_{\Sigma} u^{p}(\cdot,s)\right)^{1-\frac{1}{\chi}}\int^{t}_{b}\left(\,\int_\Sigma u^{\chi (p+m)}(\cdot,s)\right)^{\frac{1}{\chi}}\,ds\\
&\leq 2C^{-1}_1A(p,m)^{-1}\left(\frac{1}{b-a}+C_3L^m\right)^\lambda\left(\int^{t}_{a}\int_\Sigma u^{p}(\cdot,s)\,ds\right)^\lambda.
\end{split}
\end{equation*}
Thus
\begin{equation}\label{iteration1}
\norm{u}_{L^{\lambda p+m}(\Sigma\times [b,t])}\leq\left(2C^{-1}_1A(p,m)^{-1}\right)^\frac{1}{\lambda p+m}\left(\frac{1}{b-a}+C_3L^m\right)^{\frac{\lambda}{\lambda p+m}}\norm{u}^{\frac{\lambda p}{\lambda p+m}}_{L^{p}(\Sigma\times [a,t])}.
\end{equation}

Next, we perform the iteration. Fix \(t\in(0,T]\) and let
\[
q_0=p_0,\quad\  q_{i}=\lambda q_{i-1}+m\quad  \text{for} \ \ i\geq1.
\]
For \(i\geq1\), set 
\[
t_i=\left(1-2^{1-i}\right)t,
\qquad Q_i=\Sigma\times[t_i,t].
\]
Taking $p=q_{i-1}$, $b=t_{i+1}$, and $a=t_{i}$ in \eqref{iteration1}, we obtain
\begin{equation}
\norm{u}^\frac{q_i}{\lambda^{i}}_{L^{q_i}(Q_{i+1})}\leq 2^{\frac{i}{\lambda^{i-1}}}\left(2C^{-1}_1A(q_{i-1},m)^{-1}\right)^\frac{1}{\lambda^{i}}\left(t^{-1}+C_3L^m\right)^\frac{1}{\lambda^{i-1}}\norm{u}^{\frac{q_{i-1}}{\lambda^{i-1}}}_{L^{q_{i-1}}(Q_{i})}.
\end{equation}
By induction, we conclude that
\begin{equation}\label{2049}
\norm{u}^\frac{q_k}{\lambda^{k}}_{L^{q_k}(Q_{k+1})}\leq 2^{\sum^{k}_{i=1}\frac{i}{\lambda^{i-1}}}\prod^k_{i=1}\left(2C^{-1}_1A(q_{i-1},m)^{-1}\right)^\frac{1}{\lambda^{i}}\left(t^{-1}+C_3L^m\right)^{\sum^{k}_{i=1}\frac{1}{\lambda^{i-1}}}\norm{u}^{p_0}_{L^{p_0}(Q_{1})}.
\end{equation}
Note that 
\begin{equation}\label{recurrence}
\frac{q_k}{\lambda^{k}}=p_0+\frac{m}{\lambda-1}-\frac{m}{\lambda^k(\lambda-1)}.
\end{equation}
By letting $k\rightarrow\infty$ in \eqref{2049}, we obtain 
\begin{equation}\label{1117}
\norm{u(\cdot,t)}^{p_0+\frac{m}{\lambda-1}}_{L^{\infty}(\Sigma)}\leq\left(2C^{-1}_1\right)^{\frac{1}{\lambda-1}}2^{\frac{\lambda^2}{(\lambda-1)^2}}\prod_{i=1}^{\infty}A(q_{i-1},m)^{-\frac{1}{\lambda^{i}}}\left(t^{-1}+C_3L^m\right)^{\frac{\lambda}{\lambda-1}}\norm{u}^{p_0}_{L^{p_0}(Q_{1})}.
\end{equation}

It remains to estimate the infinite product
 \[\prod_{i=1}^{\infty}A(q_{i-1},m)^{-\frac{1}{\lambda^{i}}}.\]
From \eqref{recurrence}, we obtain
\[
q_k=\lambda^k p_0+\frac{m(\lambda^k-1)}{\lambda-1}.
\]
In particular, for $i\geq1$,
$$
q_{i-1}\ge \lambda^{i-1}p_0.
$$
We rewrite \(A(p,m)\) as
\[
A(p,m)=\frac{4p}{p+m}\left(1-\frac1{p+m}\right).
\]
Let $S_i=q_{i-1}+m$. Then
\[
S_i\ge \lambda^{i-1}p_0+m.
\]
Set
$$
\Theta=\prod_{i=1}^{\infty}A(q_{i-1},m)^{-\frac{1}{\lambda^{i}}}.
$$
Taking logarithms and using the factorization above gives
$$
\log \Theta
=-\log 4\sum_{i=1}^\infty \frac1{\lambda^i}
-\sum_{i=1}^\infty \frac1{\lambda^i}\log\Bigl(1-\frac{m}{S_i}\Bigr)
-\sum_{i=1}^\infty \frac1{\lambda^i}\log\Bigl(1-\frac{1}{S_i}\Bigr).
$$

Recall the elementary inequality
$$
-\log(1-x)\le \frac{x}{1-\alpha}\qquad\text{for}\ \ 0\le x\le \alpha<1.$$
Note that
$$
\frac{m}{S_i}\le \frac{m}{p_0+m}<1,
\qquad
\frac{1}{S_i}\le \frac{1}{p_0+m}<1.
$$
Therefore,
\begin{equation*}
-\log\Bigl(1-\frac{m}{S_i}\Bigr)\le \frac{m}{S_i\left(1-\frac{m}{p_0+m}\right)}=\frac{m(p_0+m)}{p_0S_i},
\end{equation*}
and
\begin{equation*}
-\log\Bigl(1-\frac{1}{S_i}\Bigr)
\le\frac{1}{S_i\left(1-\frac{1}{p_0+m}\right)}=\frac{p_0+m}{(p_0+m-1)S_i}.
\end{equation*}

Using $S_i\geq \lambda^{i-1}p_0$, we obtain
\begin{equation*}
\sum_{i=1}^\infty \frac{m}{\lambda^iS_i}
\leq\sum_{i=1}^\infty \frac{m}{p_0\lambda^{2i-1}}
= \frac{m}{p_0}\frac{1}{\lambda-\lambda^{-1}},
\end{equation*}
and similarly,
\[
\sum_{i=1}^\infty \frac1{\lambda^iS_i}\leq \frac{1}{p_0} \frac{1}{\lambda-\lambda^{-1}}.
\]
Hence
$$
\Theta \le 4^{-\frac{1}{\lambda-1}}
\exp\Biggl(
\frac{p_0+m}{p_0(\lambda-\lambda^{-1})}
\left(
\frac{m}{p_0}
+\frac{1}{p_0+m-1}
\right)
\Biggr).
$$
Since
$$
\lambda-\lambda^{-1}=\frac{4(n+1)}{n(n+2)},
$$
we conclude that
\begin{equation}\label{estpi}
\begin{split}
\Theta&\leq
2^{-n}
\exp\Biggl(
\frac{n(n+2)(p_0+m)}{4p_0(n+1)}
\left(
\frac{m}{p_0}
+\frac{1}{p_0+m-1}
\right)
\Biggr)\\
&\leq 2^{-n}
\exp\Biggl(
\frac{(n+1)(p_0+m)}{4p_0}
\left(\frac{m}{p_0}+\frac{1}{p_0+m-1}
\right)
\Biggr).
\end{split}
\end{equation}
On the other hand, it follows from \eqref{mp} that
$$\norm{u}^{p_0}_{L^{p_0}(Q_{1})}\leq tL^{p_0}.$$
Substituting this and \eqref{estpi} into \eqref{1117}, we finally obtain
\begin{equation*}
\norm{u(\cdot,t)}_{L^{\infty}}\leq C(n,m,p_0,C_S)t^{-\frac{n}{2p_0+mn}}\left(1+C_S\left(\frac{2^n}{V}\right)^{\frac{2p_0+mn}{p_0n}}L^mt\right)^{\frac{n+2}{2p_0+mn}}L^{\frac{2p_0}{2p_0+mn}},
\end{equation*}
where 
\[C\left(n,m,p_0,C_S\right)=2^{\frac{(n+2)^2}{2(2p_0+mn)}}C_S^{-\frac{n}{2p_0+mn}}\exp\left(
\frac{n+1}{4p_0}
\left(\frac{m}{p_0}+\frac{1}{p_0+m-1}
\right)
\right).\]
This completes the proof for \(n\geq 3\).
  
When $n=2$, the Sobolev inequality \eqref{sb2} allows us to apply the same Nash--Moser iteration, yielding
\begin{equation*}
\norm{u(\cdot,t)}_{L^{\infty}}\leq C(N,m,p_0,C_S)t^{-\frac{N}{2p_0+mN}}\left(1+C_{S}\left(\frac{2^N}{V}\right)^{\frac{2p_0+mN}{p_0N}}L^mt\right)^{\frac{N+2}{2p_0+mN}}L^{\frac{2p_0}{2p_0+mN}},
\end{equation*}
where 
\[C\left(N,m,p_0,C_S\right)=2^{\frac{(N+2)^2}{2(2p_0+mN)}}C_S^{-\frac{N}{2p_0+mN}}\exp\left(
\frac{N+1}{4p_0}
\left(\frac{m}{p_0}+\frac{1}{p_0+m-1}
\right)
\right),\] 
with $N=3$. This completes the proof of Theorem \ref{nm}.
\end{proof}

\begin{rem}
Define the $L^{p_0}$-to-$L^\infty$ smoothing time scale by
\[
t_0=\norm{u_0}_{L^{p_0}}^{-m}.
\]
By Theorem~\ref{nm} and the maximum principle, for all $t\geq t_0$, we have
\[
\norm{u(\cdot,t)}_{L^\infty}
\leq \norm{u(\cdot,t_0)}_{L^\infty}
\leq C_*(n,m,p_0,C_S,V)\norm{u_0}_{L^{p_0}}.
\]
\end{rem}

\section{Proofs of Theorems \ref{cobordant1} and \ref{finiteness1} and Corollary \ref{ros}}\label{3}
To prove \thmref{cobordant1}, we construct the required cobordism using the quasi-spherical metrics introduced by Bartnik \cite{Bartnik1}. The following lemma records the relevant formulas; for detailed calculations, see \cite[Lemma 2.1]{SWW}.

\begin{lem}\label{quassphericalscalar}
Let $\Sigma$ be a closed manifold, and let
$\{\gamma_t\}_{t\in[0,1]}$ be a smooth family of Riemannian
metrics on $\Sigma$. Set $W=\Sigma\times [0,1]$ and $\Sigma_t=\Sigma\times\{t\}$. For a smooth positive function
$u$ on $W$, consider the metrics
\[
\bar g=dt^2+\gamma_t,
\qquad
g=u^2dt^2+\gamma_t.
\]
Denote by $\bar A_t$ and $\bar H_t$ the second fundamental
form and mean curvature of $\Sigma_t$ in $(W,\bar g)$
with respect to the unit normal $-\partial_t$, and by
$A_t$ and $H_t$ the corresponding quantities in $(W,g)$
with respect to the unit normal $-u^{-1}\partial_t$.
Then
\begin{equation}\label{sec2meancurrelation}
A_t=u^{-1}\bar A_t,
\qquad
H_t=u^{-1}\bar H_t,
\end{equation}
and the scalar curvature of $g$ is given by
\[
R_g=u^{-2}R_{\bar g}+\left(1-u^{-2}\right)R_{\gamma_t}+2u^{-3}\bar H_t\partial_tu-2u^{-1}\Delta_{\gamma_t}u.
\]
In particular, if $u$ satisfies the quasi-spherical equation
\begin{equation}\label{quassphericalueq}
\bar H_t\partial_tu
=u^2\Delta_{\gamma_t}u
+\frac12\left(\varphi-R_{\gamma_t}\right)u^3
+\frac12\left(R_{\gamma_t}-R_{\bar g}\right)u
\end{equation}
for some smooth function $\varphi$ on $W$, then $R_g=\varphi$.
\end{lem}

\begin{proof}[Proof of Theorem \ref{cobordant1}:]
In this proof, we use the same notation as in Lemma \ref{quassphericalscalar}. Let \(f\) be a smooth positive function on \([0,1]\), to be chosen later, satisfying \(f(0)=1\). On $W$, define $\bar g=dt^2+f^2(t)\gamma.$
Direct calculations yield
\begin{equation*}
\Delta_{\gamma_t}=f^{-2}\Delta_{\gamma},\qquad R_{\gamma_t}=f^{-2}R_{\gamma},
\end{equation*}
\begin{equation}\label{qsh}
\bar{A}_t=\frac{1}{2}\partial_t\gamma_t=f'f\gamma,\qquad \bar{H}_t=\tr_{\gamma_t}\bar{A}_t=\frac{nf'}{f},
\end{equation}
and 
\begin{align*}
\partial_t\bar H_t=n\left(\frac{f''}{f}-\frac{f'^2}{f^2}\right).
\end{align*}
By the Gauss and Jacobi equations,
\begin{align*}
R_{\bar g}&=R_{\gamma_t}-\left(\bar H_t^2+|\bar A_t|^2\right)-2\partial_t \bar H_t\\
&=R_{\gamma_t}-n(n-1)\frac{f'^2}{f^2}-\frac{2nf''}{f}.
\end{align*}

We prescribe $\varphi=R_{\gamma_t}$, so that \eqref{quassphericalueq} becomes
\begin{equation}\label{qs1}
nf'f\frac{\partial u}{\partial t}=u^2\Delta_{\gamma}u+n\left(\frac{n-1}{2}f'^2+ff''\right)u.
\end{equation}
We choose $f(t)=(1+t)^\frac{2}{n+1}$, which satisfies
\[\frac{n-1}{2}f'^2+ff''=0.\]
Equation \eqref{qs1} then reduces to
$$\frac{2n}{n+1}(1+t)^{\frac{3-n}{n+1}}\frac{\partial u}{\partial t}=u^2\Delta_{\gamma}u.$$
For a smooth positive solution \(u\) of this equation, set
\[g=u^2dt^2+(1+t)^{\frac{4}{n+1}}\gamma.\] 
The scalar curvature of the neck \((W,g)\) then satisfies
\begin{equation*}
\begin{split}
R_{g}&=R_{\gamma_t}=(1+t)^{-\frac{4}{n+1}}R_{\gamma}\\
&\geq \min\left\{\min_\Sigma R_{\gamma},2^{-\frac{4}{n+1}}\min_{\Sigma} R_{\gamma}\right\}.
\end{split}
\end{equation*}
  
Now, introduce the new time variable
\[s(t)=\frac{(n+1)^2}{4n(n-1)}\left((1+t)^\frac{2(n-1)}{n+1}-1\right),\qquad t\in[0,1],\]
and denote its inverse function by $t(s)$. Define
\(\tilde{u}(s)=u(t(s)).\)
Then 
$$\partial_s \tilde{u}=\tilde{u}^2\Delta_\gamma \tilde{u}.$$
We now consider the following initial value problem on \(\Sigma\times[0,s(1)]\):
\begin{equation}\label{qs3}
\left\{
\begin{aligned}
\partial_s \tilde{u}&=\tilde u^2\Delta_{\gamma}\tilde u,\\
\tilde u(\cdot,0)&=\frac{2n}{n+1}H^{-1}.
\end{aligned}
\right.
\end{equation}

Note that 
$$\norm{\tilde u_0}_{L^{p}}=\frac{2n}{n+1}\norm{H}_{L^{-p}}^{-1}.$$
For \(n\geq3\), applying Theorem \ref{nm} with \(m=2\) and \(p_0=p\), we obtain
\begin{equation}\label{830}
\norm{\tilde u(\cdot,s)}_{L^\infty}
\leq C s^{-\frac{n}{2(p+n)}}
\norm{H}_{L^{-p}}^{-\frac{p}{p+n}}
\left(
1+
\frac{4n^2}{(n+1)^2}C_S
\left(\frac{2^n}{V}\right)^{\frac{2(p+n)}{pn}}
\norm{H}_{L^{-p}}^{-2}s
\right)^{\frac{n+2}{2(p+n)}},
\end{equation}
where
\[C=2^{\frac{(n+2)^2}{4(p+n)}}
\left(\frac{2n}{n+1}\right)^{\frac{p}{p+n}}\exp\left(
\frac{(n+1)(3p+2)}{4p^2(p+1)}
\right)C_S^{-\frac{n}{2(p+n)}}.\]
By \eqref{qsh}, we have
\begin{equation}\label{829}
H_t=\frac{2n}{(n+1)(t+1)u(\cdot,t)}=\frac{2n}{(n+1)(t+1)\tilde{u}(\cdot,s(t))}.
\end{equation}
Setting $t=1$ and substituting \eqref{830} into \eqref{829} yields
\[
H_1\geq \frac{ns(1)^{\frac{n}{2(p+n)}}\norm{H}_{L^{-p}}^{\frac{p}{p+n}}}{(n+1)C\left(
1+
\frac{4n^2}{(n+1)^2}C_S
\left(\frac{2^n}{V}\right)^{\frac{2(p+n)}{pn}}
\norm{H}_{L^{-p}}^{-2}s(1)
\right)^{\frac{n+2}{2(p+n)}}}.
\]

The right-hand side is a strictly increasing function of
$\norm{H}_{L^{-p}}$ and tends to infinity as
$\norm{H}_{L^{-p}}$ tends to infinity. Thus, there exists a positive
constant $C(n,p,\kappa,C_S,V)$, increasing in $\kappa$,
such that when $\norm{H}_{L^{-p}}\geq C(n,p,\kappa,C_S,V)$,
$ H_1\geq\kappa$.

 When \(n=2\), the same conclusion follows by applying the two-dimensional estimate in Theorem 1.4 with effective dimension \(N=3\).

Note that the induced metric at $t=1$ is
$2^{4/(n+1)}\gamma$. Setting $\widetilde H=H_1$,
we obtain the desired cobordism from $(\gamma,-H)$
to $(2^{4/(n+1)}\gamma,\widetilde H)$.

\end{proof} 

We now turn to the proof of Theorem \ref{finiteness1}.  
\begin{proof}[Proof of Theorem \ref{finiteness1}:] We use the same notation as in the proof of Theorem \ref{cobordant1}. 
Suppose that $(\Omega,g_0)$ is a fill-in of $(\Sigma,\gamma,H)$ with
$R_{g_0}\geq\sigma$, where $\sigma\leq0$ and $H>0$. Fix $T>0$. On $\Sigma\times[0,T]$, consider the metric
\[
g=u^2dt^2+(1+t)^{\frac4{n+1}}\gamma,
\]
constructed as above by prescribing
$R_g=\varphi=R_{\gamma_t}$. Then, 
\begin{equation*}
\begin{split}
R_{g}
&=(1+t)^{-\frac{4}{n+1}}R_{\gamma}\\
&\geq \min\left\{\min_{\Sigma} R_{\gamma},(1+T)^{-\frac{4}{n+1}}\min_{\Sigma} R_{\gamma}\right\}.
\end{split}
\end{equation*}

Glue $(\Omega,g_0)$ to $(\Sigma\times[0,T],g)$ along their
common boundary $(\Sigma,\gamma)$, and denote the resulting
manifold by $(\hat\Omega,\hat g)$. This gives a fill-in of
the Bartnik data  $\left(\Sigma,(1+T)^{4/(n+1)}\gamma,H_T\right)$ with corners along $(\Sigma,\gamma)$. Away from the corners, $\hat g$ is smooth and its scalar curvature is bounded below by $\min\{\sigma,\min_{\Sigma} R_{\gamma}\}$.  Along the corners, both the induced metrics and the mean curvatures from the two sides match. Then after performing the mollifying procedure in \cite{Miao02} or \cite{BH23} and a suitable conformal deformation that preserves the boundary metric, we obtain a family of
smooth metrics $\{\hat g_\delta\}_{0<\delta\leq\delta_0}$
on $\hat\Omega$ satisfying
\[
\hat g_\delta|_{\partial\hat\Omega}
=(1+T)^{\frac{4}{n+1}}\gamma,
\qquad
R_{\hat g_\delta}\geq\min\big\{\sigma,\min_{\Sigma} R_{\gamma}\big\}.
\]
Moreover, $\hat g_\delta\to\hat g$ locally uniformly away from the corners in the $C^{1,\,\alpha}$-sense. So the induced mean curvature of $\hat g_\delta$ on $\partial\hat\Omega$, denoted by $\hat H_\delta$, converges to $H_T$ as $\delta\rightarrow 0$ in the $C^\alpha$-sense. Clearly, $(\hat\Omega,\hat g_\delta)$ is a fill-in of $(\Sigma,(1+T)^{4/(n+1)}\gamma,\hat H_\delta)$ with $R_{\hat g_\delta}\geq \min\{\sigma,\min_{\Sigma} R_{\gamma}\}$. By definition, 
\begin{equation*}
\min_{\Sigma}\hat H_\delta\leq \eta\left(\Sigma,(1+T)^{\frac{4}{n+1}}\gamma,\min\left\{\sigma,\min_\Sigma R_{\gamma}\right\}\right).
\end{equation*}
Letting $\delta\rightarrow 0$, we conclude that
\begin{equation}\label{973}
\min_{\Sigma}H_{T}\leq \eta\left(\Sigma,(1+T)^{\frac{4}{n+1}}\gamma,\min\left\{\sigma,\min_\Sigma R_{\gamma}\right\}\right).
\end{equation}

We first consider the case $n\geq3$. For simplicity, let
\[
b=\frac{n+2}{2(p+n)},
\qquad D=
\frac{4n^2}{(n+1)^2}C_S
\left(\frac{2^n}{V}\right)^{\frac{2(p+n)}{pn}}.\]
Then \eqref{830} can be rewritten as 
\[
\norm{\tilde u(\cdot,s)}_{L^\infty}\leq
C s^{-\frac{n}{2(p+n)}}
\norm{H}_{L^{-p}}^{-\frac{p}{p+n}}
\left(1+sD\norm{H}_{L^{-p}}^{-2}\right)^b,
\]
where
\[
s(t)=\frac{(n+1)^2}{4n(n-1)}
\left((1+t)^{\frac{2(n-1)}{n+1}}-1\right),
\]
and 
\[C=2^{\frac{(n+2)^2}{4(p+n)}}
\exp\left(
\frac{(n+1)(3p+2)}{4p^2(p+1)}
\right)
\left(\frac{2n}{n+1}\right)^{\frac{p}{p+n}}C_S^{-\frac{n}{2(p+n)}}.\]
Combining \eqref{830}, \eqref{829}, and \eqref{973} yields
\begin{equation}\label{8311}
\frac{2ns(T)^{\frac{n}{2(p+n)}}\norm{H}_{L^{-p}}^{\frac{p}{p+n}}}{C(n+1)(1+T) \left(1+D s(T)\norm{H}_{L^{-p}}^{-2}\right)^b}\leq \eta\left(\Sigma,(1+T)^{\frac4{n+1}}\gamma,\min\left\{\sigma,\min_{\Sigma} R_{\gamma}\right\}\right).
\end{equation}

Now, choose $T$ such that $s(T)=1$, and let
\[
P=1+\frac{4n(n-1)}{(n+1)^2}.
\] 
Then 
\begin{equation}\label{eq:t0}
T=P^{\frac{n+1}{2(n-1)}}-1\qquad\text{and}\qquad (1+T)^{\frac4{n+1}}=P^{\frac{2}{n-1}}.\end{equation}
Set
\[B=\frac{1}{n}\eta\left(\Sigma,(1+T)^{\frac4{n+1}}\gamma,\min\big\{\sigma,\min_{\Sigma} R_{\gamma}\big\}\right).\]
Then \eqref{8311} becomes  
\[
\frac{2}{C(n+1)P^{\frac{n+1}{2(n-1)}}}
\norm{H}_{L^{-p}}^{\frac{p}{p+n}}
\left(1+D\norm{H}_{L^{-p}}^{-2}\right)^{-b}
\leq B,
\]
or equivalently, 
\begin{equation}\label{eq:basic-h}
\norm{H}_{L^{-p}}^{\frac{p}{p+n}}
\left(1+D\norm{H}_{L^{-p}}^{-2}\right)^{-b}
\leq R,
\end{equation}
where
\begin{equation}\label{eqR}
R=\frac{n+1}{2}
CBP^{\frac{n+1}{2(n-1)}}.
\end{equation}

\noindent{\bf Case 1: $D\norm{H}_{L^{-p}}^{-2}\leq 1$.}
In this case,
\[
1+D\norm{H}_{L^{-p}}^{-2}\leq 2,
\]
and hence
\[
\left(1+D\norm{H}_{L^{-p}}^{-2}\right)^{-b}\geq 2^{-b}.
\]
It follows from \eqref{eq:basic-h} that
\[
2^{-b}\norm{H}_{L^{-p}}^{\frac{p}{p+n}}\leq R.
\]
Consequently,
\begin{equation}\label{eq:h-case1}
\norm{H}_{L^{-p}}\leq \left(2^bR\right)^{\frac{p+n}{p}}.
\end{equation}

\noindent{\bf Case 2: $D\norm{H}_{L^{-p}}^{-2}\geq 1$.} In this case,
\[
1+D\norm{H}_{L^{-p}}^{-2}\leq 2D\norm{H}_{L^{-p}}^{-2}.
\]
Therefore,
\begin{align*}
\norm{H}_{L^{-p}}^{\frac{p}{p+n}}
\left(1+D\norm{H}_{L^{-p}}^{-2}\right)^{-b}
&\geq
\norm{H}_{L^{-p}}^{\frac{p}{p+n}}
\left(2D\norm{H}_{L^{-p}}^{-2}\right)^{-b}
\\
&=(2D)^{-b}\norm{H}_{L^{-p}}^{\frac{p+n+2}{p+n}}.
\end{align*}
Using \eqref{eq:basic-h}, we conclude that
\begin{equation}\label{eq:h-case2}
\norm{H}_{L^{-p}}\leq
\left((2D)^bR\right)^{\frac{p+n}{p+n+2}}.
\end{equation}
Combining \eqref{eq:h-case1} and \eqref{eq:h-case2} yields
\begin{equation}\label{eq:h-compact}
\norm{H}_{L^{-p}}\leq
\max\left\{
\left(2^bR\right)^{\frac{p+n}{p}},
\left((2D)^bR\right)^{\frac{p+n}{p+n+2}}
\right\}.
\end{equation}
Substituting \eqref{eqR} into \eqref{eq:h-compact} and noting that
\begin{align*}
C\leq
2^{\frac{(n+2)^2+4p}{4(p+n)}}C_S^{-\frac{n}{2(p+n)}}\exp\left(
\frac{(n+1)(3p+2)}{4p^2(p+1)}
\right),
\end{align*}
and
\begin{align*}
D\leq 4C_S
\left(\frac{2^n}{V}\right)^{\frac{2(p+n)}{pn}},
\end{align*}
we conclude that 
\begin{equation*}\label{eq:h-explicit}
\begin{aligned}
\norm{H}_{L^{-p}}
\leq &\,\max\Bigg\{
\left(2^{\frac{n^2+2n+8}{4}}\left((n+1)P^{\frac{n+1}{2(n-1)}}\right)^{n+p}\exp\left(
\frac{(n+1)(3p+2)(p+n)}{4p^2(p+1)}
\right)C_S^{-\frac{n}{2}}B^{p+n}\right)^{\frac{1}{p}},
\\
&
\left(2^{\frac{n^2+6n+16}{4}}\left((n+1)P^{\frac{n+1}{2(n-1)}}\right)^{n+p}
\exp\left(
\frac{(n+1)(3p+2)(p+n)}{4p^2(p+1)}
\right)C_S\left(\frac{2^n}{V}\right)^{
\frac{(n+2)(p+n)}{pn}}B^{p+n}\right)^{\frac{1}{p+n+2}}
\Bigg\}.
\end{aligned}
\end{equation*}
Using $P^{(n+1)/(2(n-1))}<3$ and $(3p+2)/(p+1)<3$, we obtain
\begin{equation*}
\begin{aligned}
\norm{H}_{L^{-p}}
\leq &\max\Bigg\{
\left(2^{\frac{(n+2)^2}{4}}\left(3(n+1)\right)^{n+p}
C_S^{-\frac{n}{2}}\exp\left(
\frac{3(n+1)(p+n)}{4p^2}
\right)B^{p+n}\right)^{\frac{1}{p}},
\\
&
\left(2^{\frac{(n+4)^2}{4}}\left(3(n+1)\right)^{n+p}
C_S\exp\left(
\frac{3(n+1)(p+n)}{4p^2}
\right)\left(\frac{2^n}{V}\right)^{
\frac{(n+2)(p+n)}{pn}}B^{p+n}\right)^{\frac{1}{p+n+2}}
\Bigg\}\\
\leq&\,\bar{C}(n,p)\max\Bigg\{\left(C_S^{-\frac{n}{2}}B^{p+n}\right)^{\frac{1}{p}}, \left(C_{S}\left(\frac{2^n}{V}\right)^{
\frac{(n+2)(p+n)}{pn}}B^{p+n}\right)^{\frac{1}{p+n+2}}
\Bigg\},
\end{aligned}
\end{equation*}
where 
\[\bar{C}(n,p)=2^{\frac{(n+4)^2}{4p}}\left(3(n+1)\right)^{\frac{n+p}{p}}\exp\left(\frac{3(n+1)(p+n)}{4p^3}\right).\]
Hence, 
\begin{equation*}
\int_{\Sigma}\frac{1}{H^{p}} \,d\mu_{\gamma}\geq C(n,p)\min\Bigg\{\left(C_S^{\frac{n}{2}}B^{-(p+n)}\right),\left(C_{S}\left(\frac{2^n}{V}\right)^{
\frac{(n+2)(n+p)}{pn}}B^{p+n}\right)^{-\frac{p}{p+n+2}}
\Bigg\},
\end{equation*}
where
\begin{equation}\label{99}
C(n,p)=2^{-\frac{(n+4)^2}{4}}\bigl(3(n+1)\bigr)^{-(n+p)}\exp\left(-\frac{3(n+1)(p+n)}{4p^2}\right)
\end{equation}

Note that by scaling, 
\begin{equation}\label{98}
\begin{split}
B&=\frac{1}{n}\eta\left(\Sigma,(1+T)^{\frac4{n+1}}\gamma,\min\left\{\sigma,\min_\Sigma R_{\gamma}\right\}\right)\\
&=\frac{1}{n}P^{-\frac1{n-1}}\eta\left(\Sigma,\gamma,P^{\frac2{n-1}}\min\left\{\sigma,\min_\Sigma R_{\gamma}\right\}\right).\end{split}
\end{equation}
By letting 
\begin{equation}\label{910}\tau(n)=\frac{1}{n}P^{-\frac1{n-1}}=\frac{1}{n}\left(1+\frac{4n(n-1)}{(n+1)^2}\right)^{-\frac1{n-1}},\end{equation}
we complete the proof when $n\geq3$. 

When $n=2$, by Theorem \ref{nm}, the same bound holds with $n$ replaced by $3$ in the estimate and in $C(n,p)$. The definition of $B$ retains the actual dimension $n=2$. 

We complete the proof of Theorem \ref{finiteness1}.

\end{proof}

\begin{proof}[Proof of Corollary \ref{ros}:] 
For convenience, we consider the following invariant for spin fill-ins. For \(\sigma\in\mathbb R\), define 
\[
\eta^{\mathrm{spin}}(\Sigma,\gamma,\sigma)=\sup\left\{\min_{\Sigma} H\,\big|\,(\Sigma,\gamma,H)\text{ admits a spin fill-in with scalar curvature $R\geq\sigma$}\right\}.
\]
 Gromov \cite{Gro18} proved that 
\begin{equation*}
\eta^{\mathrm{spin}}(\Sigma,\gamma,0)
\leq
\frac{n}{\mathrm{Rad}(\Sigma,\gamma)}.
\end{equation*}
More recently, combining B\"ar's eigenvalue estimate \cite{B0}
with the Hijazi--Montiel--Rold\'an inequality \cite{HMR},
Brendle, Tsiamis, and Wang \cite{BTW} showed that
\begin{equation*}
\eta^{\mathrm{spin}}(\Sigma,\gamma,-n(n+1))\leq
n\left(1+\mathrm{Rad}\left(\Sigma,\gamma\right)^{-2}\right)^{\frac{1}{2}}.
\end{equation*}
By scaling, these estimates imply that
\begin{align}\label{sca}
\begin{split}
&\eta^{\mathrm{spin}}\left(\Sigma,\left(1+\frac{4n(n-1)}{(n+1)^2}\right)^{\frac{2}{n-1}}
\gamma,\min\left\{\sigma,\min_{\Sigma}R_{\gamma}\right\}\right)\\
\leq&\, n\left(\frac{-\min\left\{\sigma, \min_{\Sigma}R_{\gamma}\right\}}{n(n+1)}+\mathrm{Rad}\left(\Sigma,\left(1+\frac{4n(n-1)}{(n+1)^2}\right)^{\frac{2}{n-1}}\gamma\right)^{-2}\right)^{\frac{1}{2}}\\
=&\,n\left(\frac{-\min\left\{\sigma, \min_{\Sigma}R_{\gamma}\right\}}{n(n+1)}+\left(1+\frac{4n(n-1)}{(n+1)^2}\right)^{-\frac{2}{n-1}}\mathrm{Rad}\left(\Sigma,\gamma\right)^{-2}\right)^{\frac{1}{2}}\\
\leq&\, n\left(\frac{-\min\left\{\sigma, \min_{\Sigma}R_{\gamma}\right\}}{n(n+1)}+\mathrm{Rad}\left(\Sigma,\gamma\right)^{-2}\right)^{\frac{1}{2}}.
\end{split}
\end{align}
Since attaching the product collar preserves the spin
structure, the proof of Theorem \ref{finiteness1}
applies with $\eta$ replaced by $\eta^{\mathrm{spin}}$.
Thus, the resulting integral estimate holds with
\[
B=\frac1n\eta^{\mathrm{spin}}\left(\Sigma,\left(1+\frac{4n(n-1)}{(n+1)^2}\right)^{\frac{2}{n-1}}
\gamma,\min\left\{\sigma,\min_{\Sigma}R_{\gamma}\right\}\right).
\]
Using \eqref{sca}, together with the
monotonicity of the right-hand side of that estimate
with respect to $B$, yields the desired conclusion.

\end{proof}
 
\section{Proofs of Theorem \ref{cobordant2} and Corollary \ref{finiteness2}}\label{4}

In this section, we prove Theorem \ref{cobordant2} and Corollary \ref{finiteness2}. Using the homothetic path constructed in the previous section, we prescribe the scalar curvature so that the corresponding quasi-spherical equation contains an absorbing cubic term. This yields a positive lower bound for the terminal mean curvature that is independent of the initial mean curvature.
\begin{proof}[Proof of Theorem \ref{cobordant2}:]
As in the proof of Theorem \ref{cobordant1}, on \(\Sigma\times[0,1]\), set
\[
 f(t)=(1+t)^{\frac{2}{n+1}},
 \qquad
 \gamma_t=f(t)^2\gamma,
\]
and consider the metric
\[
 \overline g=dt^2+\gamma_t.
\]
Then, 
\[
 \Delta_{\gamma_t}=f^{-2}\Delta_\gamma,
 \qquad
 R_{\gamma_t}=f^{-2}R_\gamma,
 \qquad
 \overline H_t=\frac{nf'}{f}
 =\frac{2n}{(n+1)(1+t)},
\]
and our choice of \(f\) gives
\[
 R_{\overline g}=R_{\gamma_t}=f^{-2}R_\gamma.
\]

For simplicity, let
\[\lambda=\frac{1}{2}\left(\min_{\Sigma} R_{\gamma}-\delta\right),\]
and let
\begin{equation*}
\varphi=f^{-2}\big(R_\gamma-2\lambda\big).
\end{equation*}
It is easy to see that 
\[\varphi\geq \min\left\{\delta, 2^{-\frac{4}{n+1}}\delta\right\}.\]
The quasi-spherical equation \eqref{quassphericalueq} for \(R_{g}=\varphi\) reduces to
\begin{equation}\label{qseqn}
\left\{
\begin{aligned}
nf f'\frac{\partial u}{\partial t}
 &=u^2\Delta_\gamma u-\lambda u^3,\\
 u(\cdot,0)&=\frac{2n}{n+1}H^{-1}.
\end{aligned}
\right.
\end{equation}
As in the proof of Theorem \ref{cobordant1}, define
\begin{equation}\label{963}
 s(t)=\frac{(n+1)^2}{4n(n-1)}
 \left[(1+t)^{\frac{2(n-1)}{n+1}}-1\right].
\end{equation}
Writing \(u(s)=u(t(s))\) for simplicity, \eqref{qseqn} becomes
\begin{equation}\label{qseqn1}
 \frac{\partial u}{\partial s}
 =u^2\Delta_\gamma u-\lambda u^3.
\end{equation}
Let
\[
 M(s)=\max_\Sigma u(\cdot,s),
 \qquad
 m(s)=\min_\Sigma u(\cdot,s).
\]
Similar to the proof of \cite[Lemma 2.1]{SWW}, by comparison with the ODE
\(v'=-\lambda v^3\) and using the maximum principle, we obtain 

\begin{equation}\label{ma}
 M(s)\leq
 \big(M(0)^{-2}+2\lambda s\big)^{-\frac12}
 \leq(2\lambda s)^{-\frac12}
\end{equation}
for \(s>0\), while
\[
 m(s)\geq
 \big(m(0)^{-2}+2\lambda s\big)^{-\frac12}>0.
\]
These estimates show that the positive solution of
\eqref{qseqn1} exists throughout the required time
interval.

Now, consider a quasi-spherical metric \(
 g=u^2dt^2+\gamma_t\). 
Denote by $H_t$ the mean curvature of the slice
$\Sigma\times\{t\}$ in $(\Sigma\times[0,1],g)$ with respect
to the unit normal $u^{-1}\partial_t$.
By \eqref{sec2meancurrelation}, we have $H_0=H$.
At $t=1$, combining \eqref{qseqn1} and \eqref{ma} yields
\begin{align*}
 H_1
 &\geq \frac{n}{(n+1)M(s(1))}\\
 &\geq \frac{n}{n+1}\sqrt{2\lambda s(1)}\\
  &=\left[
 \frac{n}{4(n-1)}
 \left(2^{\frac{2(n-1)}{n+1}}-1\right)\left(\min_{\Sigma} R_{\gamma}-\delta\right) \right]^{\frac12}.
\end{align*}

We next compare the total mean curvatures.  Let
\[
 J(s)=\int_\Sigma u(\cdot,s)^{-1}\,d\mu_\gamma.
\]
It follows from \eqref{qseqn1} that
\begin{align*}
 J'(s)
 &=-\int_\Sigma u^{-2}\frac{\partial u}{\partial s}
 \,d\mu_\gamma\\
 &=-\int_\Sigma\Delta_\gamma u\,d\mu_\gamma
 +\lambda\int_\Sigma u\,d\mu_\gamma\\
 &=\lambda\int_\Sigma u\,d\mu_\gamma\geq0.
\end{align*}
Since \(d\mu_{\gamma_t}=f(t)^n d\mu_\gamma\), we have
\begin{equation}\label{H1}
 \int_\Sigma H_t\,d\mu_{\gamma_t}
 =\frac{2n}{n+1}(1+t)^{\frac{n-1}{n+1}}J(s(t)).
\end{equation}
Combining the monotonicity of \(J\)
 and
\eqref{H1}, we obtain
\begin{equation*}
\begin{split}
 \int_\Sigma H\,d\mu_\gamma
 \leq&
 2^{-\frac{n-1}{n+1}}
 \int_\Sigma H_1\,d\mu_{\gamma_1}\\
 =&2\int_\Sigma H_1\,d\mu_{\gamma}.
 \end{split}
\end{equation*}
By letting $\widetilde H=H_1$, we complete the proof of Theorem \ref{cobordant2}.
\end{proof}
Now, we are ready to provide the proof of Corollary \ref{finiteness2}.

\begin{proof}[Proof of Corollary \ref{finiteness2}:]
Assume $(\Omega^{n+1},g_0)$ is a fill-in of Bartnik data $(\Sigma^{n},\gamma,H)$ with scalar curvature $R_{g_0}\geq \sigma$ and $H\geq0$. We first deform the metric so that the boundary mean curvature
becomes strictly positive as that in \cite{FHH}. By the boundary deformation argument
in \cite[Proposition 3.8]{BH23}, for any sufficiently small
$\epsilon\in(0,1)$, we may modify $g_0$ in an arbitrarily small
neighborhood of $\partial\Omega$ to obtain a metric $g_\epsilon$
satisfying
\[
g_\epsilon|_{\partial\Omega}=\gamma,\qquad
H_\epsilon>H\geq0,\qquad
R_{g_\epsilon}\geq\sigma-\epsilon\geq\sigma-1.
\]
In particular,
\begin{equation}\label{912}
\int_\Sigma H\,d\mu_\gamma
\leq
\int_\Sigma H_\epsilon\,d\mu_\gamma.
\end{equation}
By choosing $\delta=\min_{\Sigma} R_{\gamma}-1$ in Theorem \ref{cobordant2}, we know there exists a cobordism $g$ from $\left(\gamma,-H_{\epsilon}\right)$
to $\left(2^{4/(n+1)}\gamma,\widetilde H\right)$ with  
\begin{equation}\label{H2}
 R_g\geq\min\left\{\delta, 2^{-\frac{4}{n+1}}\delta\right\},\qquad \widetilde H\geq 
 \left[
 \frac{n}{4(n-1)}
 \left(2^{\frac{2(n-1)}{n+1}}-1\right)
 \right]^{\frac12}>0,
\end{equation}
and
\begin{equation}\label{91}
 \int_\Sigma H_{\epsilon}\,d\mu_\gamma
 \leq
 2^{-\frac{n-1}{n+1}}
 \int_\Sigma \widetilde H\,d\mu_{2^{4/(n+1)}\gamma}.
 \end{equation} 
 
 Now, by gluing $(\Omega,g_{\epsilon})$ with this cobordism along the common boundary $(\Sigma,\gamma)$, we obtain a fill-in of $\left(\Sigma,2^{4/(n+1)}\gamma,\widetilde H\right)$ with corner along $(\Sigma,\gamma)$ and with scalar curvature not less than 
 \[\min\left\{\sigma-1, \delta, 2^{-\frac{4}{n+1}}\delta\right\},\]
 away from the corner. By performing the same mollifying procedure as in the proof of Theorem \ref{finiteness1}, we may find $\alpha_0>0$ such that,
for every $0<\alpha\leq\alpha_0$, there exists a smooth
fill-in $(\hat\Omega,\hat g_\alpha)$ of
$(\Sigma,2^{4/(n+1)}\gamma,\hat H_\alpha)$ satisfying\[
R_{\hat g_\alpha}\geq
\min\left\{\sigma-1,\delta,2^{-\frac{4}{n+1}}\delta\right\},
\qquad
\hat H_\alpha\rightarrow\widetilde H
\quad\text{as }\alpha\to0.
\]
Choosing $\alpha_0$ sufficiently small, we may assume that, for every $0<\alpha\leq\alpha_0$,
\begin{equation}\label{H2}
\hat H_\alpha\geq
\frac12\left[
\frac{n}{4(n-1)}
\left(2^{\frac{2(n-1)}{n+1}}-1\right)
\right]^{\frac12}>0.
\end{equation}
It then follows from \cite[Theorem A]{FHH} that
\begin{equation*}
\int_\Sigma \hat H_\alpha\,d\mu_{2^{\frac{4}{n+1}}\gamma}
\leq C\Bigg(
\Sigma,2^{\frac{4}{n+1}}\gamma,
\min\left\{\sigma-1,\delta,2^{-\frac{4}{n+1}}\delta\right\},
\frac12\left[
\frac{n}{4(n-1)}
\left(2^{\frac{2(n-1)}{n+1}}-1\right)
\right]^{\frac12}
\Bigg).
\end{equation*}
The right-hand side depends only on $\Sigma$, $\gamma$ and $\sigma$. By
letting $\alpha\to0$, we obtain
\[
\int_\Sigma \widetilde H\,d\mu_{2^{\frac{4}{n+1}}\gamma}
\leq C(\Sigma,\gamma,\sigma).
\]
Substituting this estimate into \eqref{91} and using
\eqref{912}, we conclude that
\[
\int_\Sigma H\,d\mu_\gamma
\leq C(\Sigma,\gamma,\sigma).
\]
This completes the proof of Corollary \ref{finiteness2}.\end{proof}

\medskip

\subsection*{Acknowledgement}
The authors asked Jingang Xiong whether the  \(L^{p}\)-to-\(L^\infty\) estimate for \(\eqref{maineq}\) could be strengthened to an \(L^{-1}\)-to-\(L^\infty\) estimate. He provided a counterexample showing that such a strengthening is impossible. The authors are grateful to him for this observation and for other stimulating discussions.

\subsection*{Development of the Results and the Use of Generative AI}
The authors began this project in 2022 and obtained preliminary versions of Theorem \ref{cobordant1}, \ref{nm} and \ref{finiteness1} later that year. They subsequently sought to strengthen the \(L^{p}\)-to-\(L^\infty\) estimate to an \(L^{-1}\)-to-\(L^\infty\) estimate for \(\eqref{maineq}\), with the aim of proving the \(H\geq0\) case of Conjecture \ref{conj}. However, Jingang Xiong’s counterexample ruled out such a strengthening. Recently, the authors realized that Theorem \ref{cobordant2} can be used to reduce the \(H\geq0\) case of the conjecture to part (A) of Theorem \ref{fhh}, due to Frenck, Hanke, and Hirsch \cite{FHH}.

All mathematical ideas, arguments, and results presented in this manuscript were developed independently by the authors, without the use of generative AI. Following completion of the initial draft, AI tools were used only to improve the language. The authors take full responsibility for the content of the manuscript.

\end{document}